\documentclass{article}%
\usepackage{amsfonts}
\usepackage{amssymb}
\usepackage{graphicx}
\usepackage{amsmath}
\usepackage{amscd}%
\newtheorem{theorem}{Theorem}

\newtheorem{corollary}[theorem]{Corollary}

\newtheorem{definition}[theorem]{Definition}

\newtheorem{lemma}[theorem]{Lemma}

\newenvironment{proof}[1][Proof]{\textbf{#1.} }{\ \rule{0.5em}{0.5em}}
\begin{document}

\title{Interlacement in 4-regular graphs: a new approach using nonsymmetric matrices}
\author{Lorenzo Traldi\\Lafayette College\\Easton, Pennsylvania 18042}
\maketitle
\date{}

\begin{abstract}
Let $F$ be a 4-regular graph with an Euler system $C$. We introduce a simple
way to modify the interlacement matrix of $C$ so that every circuit partition
$P$ of $F$ has an associated modified interlacement matrix $M(C,P)$. If $C$
and $C^{\prime}$ are Euler systems of $F$ then $M(C,C^{\prime})$ and
$M(C^{\prime},C)$ are inverses, and for any circuit partition $P$,
$M(C^{\prime},P)=M(C^{\prime},C)\cdot M(C,P)$. This machinery allows for short
proofs of several results regarding the linear algebra of interlacement.

\bigskip


\bigskip

Mathematics Subject\ Classification. 05C31

\end{abstract}

\section{Interlacement and local complements}

A \emph{graph} $G=(V(G),E(G))$ is given by a finite set $V(G)$ of
\emph{vertices}, and a finite set $E(G)$ of \emph{edges}. In a \emph{looped
simple graph} each edge is incident on one or two vertices, and different
edges have different vertex-incidences;\ an edge incident on only one vertex
is a \emph{loop}. A\emph{ simple graph} is a looped simple graph with no loop.
In general, a graph may have \emph{parallel edges} (distinct edges with the
same vertex-incidences). Edge-vertex incidences generate an equivalence
relation on $E(G)\cup V(G)$; the equivalence classes are the \emph{connected
components} of $G$, and the number of connected components is denoted $c(G)$.
Two vertices incident on a non-loop edge are \emph{neighbors}, and if $v\in
V(G)$ then $N(v)=\{$neighbors of $v\}$ is the \emph{open neighborhood} of $v$.

Each edge consists of two distinct \emph{half-edges}, and the edge has two
distinct \emph{directions} given by designating one half-edge as initial and
the other as terminal. Each half-edge is incident on a vertex; if the edge is
not a loop then the half-edges are incident on different vertices. The number
of half-edges incident on a vertex $v$ is the \emph{degree} of $v$, and
a\emph{ }$d$\emph{-regular} graph is one whose vertices all have degree $d$.
In a directed graph each vertex has an \emph{indegree} and an \emph{outdegree}%
; a \emph{d-in, d-out} digraph is one whose vertices all have indegree $d$ and
outdegree $d$. A \emph{circuit} in a graph is a sequence $v_{1}$, $h_{1}$,
$h_{1}^{\prime}$, $v_{2}$, ..., $v_{k}$, $h_{k}$, $h_{k}^{\prime}$,
$v_{k+1}=v_{1}$ such that for each $i$, $h_{i+1}$ and $h_{i}^{\prime}$ are
half-edges incident on $v_{i+1}$, and $h_{i}$ and $h_{i}^{\prime}$ are the
half-edges of an edge $e_{i}$; $e_{i}\neq e_{j}$ when $i\neq j$. A
\emph{directed} circuit in a directed graph is a circuit in which $h_{i}$\ is
the initial half-edge of $e_{i}$, for every $i$. An \emph{Euler circuit} is a
circuit in which every edge appears exactly once; more generally, an
\emph{Euler system} is a collection of Euler circuits, one in each connected
component of the graph. A graph has Euler systems if and only if every vertex
is of even degree; we refer to Fleischner's books \cite{F1, F2} for the
general theory of Eulerian graphs.

In this paper we are concerned with the theory of Euler systems in 4-regular
graphs, introduced by Kotzig \cite{K}. If $v$ is a vertex of a 4-regular graph
$F$ and $C$ is an Euler system of $F$, then the $\kappa$\emph{-transform}
$C\ast v$ is the Euler system obtained from $C$ by reversing one of the two
$v$-to-$v$ walks within the circuit of $C$ incident on~$v$. \textit{Kotzig's
theorem} is that all Euler systems of $F$ can be obtained from any one using
finite sequences of $\kappa$-transformations.

The \textit{interlacement graph} $\mathcal{I}(C)$ of a 4-regular graph $F$
with respect to an Euler system $C$ was introduced by Bouchet \cite{Bold} and
Read and Rosenstiehl \cite{RR}.

\begin{definition}
$\mathcal{I}(C)$ is the simple graph with $V(\mathcal{I}(C))=V(F)$, in which
$v$ and $w$ are adjacent if and only if they appear in the order
$v...w...v...w...$ on one of the circuits of $C$.
\end{definition}

There is a natural way to construct $\mathcal{I}(C\ast v)$ from $\mathcal{I}%
(C)$.

\begin{definition}
\label{lc} Let $G$ be a simple graph, and suppose $v\in V(F)$. The
\emph{simple local complement }$G^{v}$ is the graph obtained from $G$ by
reversing adjacencies between neighbors of $v$.
\end{definition}

That is, if $w\neq x\in V(G)=V(G^{v})$ then $w$ and $x$ are neighbors in
$G^{v}$ if and only if either (a) at least one of them is not a neighbor of
$v$, and they are neighbors in $G$; or (b) both are neighbors of $v$, and they
are not neighbors in $G$. The well-known equality $\mathcal{I}(C\ast
v)=\mathcal{I}(C)^{v}$ follows from the fact that reversing one of the two
$v$-to-$v$ walks within the incident circuit of $C$ has the effect of toggling
adjacencies between vertices that appear once apiece on this walk.

Another way to describe simple local complementation involves the following.

\begin{definition}
The \emph{Boolean adjacency matrix} of a graph $G$ is the symmetric
$V(G)\times V(G)$ matrix $\mathcal{A}(G)$ with entries in $GF(2)$ given by: a
diagonal entry is 1 if and only if the corresponding vertex is looped in $G$,
and an off-diagonal entry is 1 if and only if the corresponding vertices are
neighbors in $G$.
\end{definition}

\begin{definition}
\label{lc2}Suppose $G$ is a simple graph and%
\[
\mathcal{A}(G)=%
\begin{bmatrix}
0 & \mathbf{1} & \mathbf{0}\\
\mathbf{1} & M_{11} & M_{12}\\
\mathbf{0} & M_{21} & M_{22}%
\end{bmatrix}
,
\]
with the first row and column corresponding to $v$. Then $G^{v}$ is the simple
graph whose adjacency matrix is%
\[
\mathcal{A}(G^{v})=%
\begin{bmatrix}
0 & \mathbf{1} & \mathbf{0}\\
\mathbf{1} & \overline{M}_{11}-I & M_{12}\\
\mathbf{0} & M_{21} & M_{22}%
\end{bmatrix}
\]
where $I$ is an identity matrix and the overbar indicates toggling of all entries.
\end{definition}

Kotzig's theorem tells us that the Euler systems of a 4-regular graph $F$ form
an orbit under $\kappa$-transformations. It follows that the interlacement
graphs of Euler systems of $F$ form an orbit under simple local
complementation. From a combinatorial point of view this \textquotedblleft
naturality\textquotedblright\ of interlacement graphs is intuitively
satisfying: the Euler systems of $F$ must share some structural features, as
they coexist in $F$, and these shared structural features are reflected in
shared structural features of their interlacement graphs. Many researchers
have studied simple local complementation in the decades since Kotzig founded
the theory; the associated literature is large and quite fascinating. We do
not presume to summarize this body of work, but we might mention that
intrinsic properties distinguish the simple graphs that arise as interlacement
graphs from those that do not \cite{Bco, F} and that 4-regular graphs with
isomorphic interlacement graphs are closely related to each other \cite{Gh}.

In contrast, the algebraic properties of interlacement graphs are \emph{not}
intuitively satisfying. The adjacency matrices of the various interlacement
graphs\ associated to $F$ have little in common, aside from the fact that they
are symmetric matrices of the same size. To say the same thing in a different
way, simple local complementation changes fundamental algebraic properties of
the adjacency matrix. For instance the ranks of $\mathcal{A}(G)$ and
$\mathcal{A}(G^{v})$ may be quite different; this rank change is caused by the
$-I$ in Definition \ref{lc2}.

The purpose of this paper is to present \emph{modified interlacement
matrices}, whose algebraic properties are in many ways more natural than those
of interlacement matrices. We present the theory of these matrices in\ Section
2, and then briefly summarize the connections between this theory and earlier
work in\ Section 3. Before going into detail we would like to thank R.
Brijder, H. J. Hoogeboom, D. P. Ilyutko, V. O. Manturov and L. Zulli for many
discussions of their work on interlacement, including \cite{BH2}, \cite{IM1}
and \cite{Z}. We are also grateful to an anonymous reader for comments on an
earlier version of the paper.

\section{Modified interlacement and local complements}

Our modifications involve the following notions. If $v$ is a vertex of a
4-regular graph $F$ then Kotzig \cite{K} observed that there are three
\emph{transitions} at $v$, i.e., three different pairings of the four incident
half-edges into disjoint pairs. If $C$ is an Euler system of $F$ then we can
classify these three transitions according to their relationship with $C$, as
in \cite{T5, T6}. One transition appears in $C$; we label this one $\phi$, for
\emph{follow}. Of the other two transitions, one is consistent with an
orientation of the incident circuit of $C$, and the other is not; we label
them $\chi$ and $\psi$, respectively. (It does not matter which orientation of
this circuit of $C$ is used.)

See Figure \ref{lalgint1}, where circuits are indicated with this convention:
when a circuit traverses a vertex, the dash pattern is maintained. In more
complex diagrams like Figure \ref{lalgint3} it is sometimes necessary to
change the dash pattern in the middle of an edge, in order to make sure that
two different dash patterns appear at each vertex.%

\begin{figure}
[ptb]
\begin{center}
\includegraphics[
trim=1.068682in 8.029093in 1.071155in 1.341391in,
height=1.0222in,
width=4.6069in
]%
{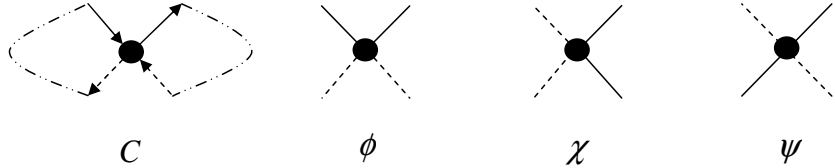}%
\caption{The three transitions at $v$ are labeled according to their
relationship with an Euler system $C$.}%
\label{lalgint1}%
\end{center}
\end{figure}

There are $3^{\left\vert V(G)\right\vert }$ different ways to choose $\phi$,
$\chi$ or $\psi$ at each vertex of $F$. Each system of choices determines a
\textit{circuit partition} or \textit{Eulerian partition }of $F$, i.e., a
partition of $E(F)$ into edge-disjoint circuits. The circuit partitions that
include precisely $c(F)$ circuits are the Euler systems of $F$. Circuit
partitions have received a great deal of attention since they were introduced
by Kotzig \cite{K}, who called them $\xi$\emph{-decompositions}. Building on
earlier work of Martin \cite{Ma}, Las Vergnas \cite{L2, L1, L} introduced the
idea of using the generating function $\sum x^{\left\vert P\right\vert }$ that
records the sizes of the circuit partitions of $F$ as a structural invariant
of $F$. This idea has subsequently appeared in knot theory (where it underlies
the Kauffman bracket \cite{Kau}) and in general graph theory (where it
motivates the interlace polynomials of Arratia, Bollob\'{a}s and Sorkin
\cite{A1, A2, A}).

Here is the central definition of the paper.

\begin{definition}
\label{modint}Let $C$ be an Euler system of a 4-regular graph $F$, and let $P$
be a circuit partition of $F$. Then the \emph{modified interlacement matrix of
}$C$\emph{ with respect to }$P$ is the matrix $M(C,P)$ obtained from
$\mathcal{A}(\mathcal{I}(C))$ by making the following changes:

\begin{enumerate}
\item If $P$ involves the $\phi$ transition with respect to $C$ at a vertex
$v$, then change the diagonal entry corresponding to $v$ to 1 and change every
other entry in that column to 0.

\item If $P$ involves the $\psi$ transition with respect to $C$ at a vertex
$v$, then change the diagonal entry corresponding to $v$ to 1.
\end{enumerate}
\end{definition}

Definition \ref{modint} might seem complicated and unmotivated, but it has the
surprising virtue that the modified interlacement matrices of different\ Euler
systems with respect to a given circuit partition are related to each other
through elementary row operations. Consequently these modified interlacement
matrices share many algebraic properties -- for instance, they all have the
same rank and the same right nullspace -- and familiar ideas of elementary
linear algebra can be used to explain these properties.

\begin{definition}
\label{lc3} Let $G$ be a graph, and let $M$ be a matrix whose rows and columns
are indexed by $V(G)$. Suppose $v\in V(G)$ and $M$ is
\[
M=%
\begin{bmatrix}
d_{vv} & \rho_{1} & \rho_{2}\\
\kappa_{1} & M_{11} & M_{12}\\
\kappa_{2} & M_{21} & M_{22}%
\end{bmatrix}
,
\]
where the first row and column correspond to $v$, the rows and columns of
$M_{11}$ correspond to vertices in $N(v)$, and the rows and columns of
$M_{22}$ correspond to vertices in $V(G)-N(v)-\{v\}$. Then the \emph{modified
local complement} of $M$ with respect to $v$ is the matrix obtained from $M$
by adding the $v$ row to every row corresponding to a neighbor of $v$:%
\[
M_{\operatorname{mod}}^{v}=%
\begin{bmatrix}
d_{vv} & \rho_{1} & \rho_{2}\\
\kappa_{1}^{\prime} & M_{11}^{\prime} & M_{12}^{\prime}\\
\kappa_{2} & M_{21} & M_{22}%
\end{bmatrix}
.
\]

\end{definition}%

\begin{figure}
[p]
\begin{center}
\includegraphics[
trim=1.068682in 6.289777in 1.207214in 1.343530in,
height=2.3263in,
width=4.5057in
]%
{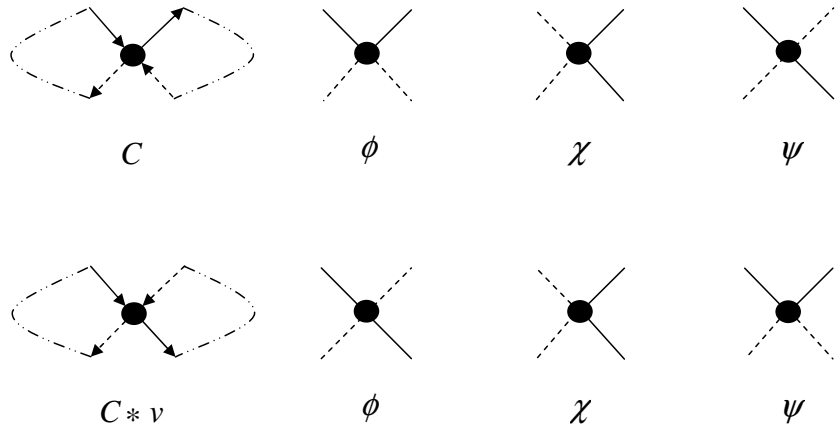}%
\caption{When $C$ is replaced with $C\ast v$, the $\phi$ and $\psi$ transition
labels are interchanged at $v$.}%
\label{lalgint2}%
\end{center}
\end{figure}
\begin{figure}
[p]
\begin{center}
\includegraphics[
trim=1.062085in 5.752793in 1.074454in 0.939188in,
height=3.0329in,
width=4.6095in
]%
{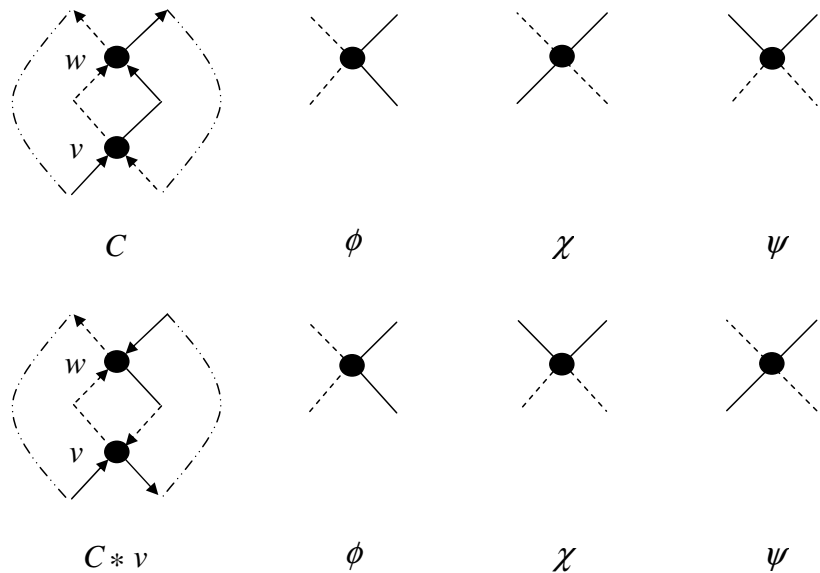}%
\caption{When $C$ is replaced with $C\ast v$, the $\chi$ and $\psi$ transition
labels are interchanged at vertices interlaced with $v$.}%
\label{lalgint3}%
\end{center}
\end{figure}

\begin{theorem}
\label{thmone}Let $C$ be an Euler system of a 4-regular graph $F$, let $P$ be
a circuit partition of $F$, and let $v\in V(F)$. Consider $M(C,P)$ to have row
and column indices from $V(\mathcal{I}(C))$. Then%
\[
M(C,P)_{\operatorname{mod}}^{v}=M(C\ast v,P).
\]

\end{theorem}

\begin{proof}
Let $\vec{v}\in GF(2)^{V(F)}$ be the vector whose only nonzero coordinate
corresponds to $v$, and let $\vec{N}(v)\in GF(2)^{V(F)}$ be the vector whose
$w$ coordinate is 1 if and only if $w$ neighbors $v$ in $\mathcal{I}(C)$ and
$\mathcal{I}(C\ast v)$.

We first verify that $M(C\ast v,P)$ and $M(C,P)_{\operatorname{mod}}^{v}$ have
the same $v$ column. As illustrated in Figure \ref{lalgint2}, if $P$ involves
the $\phi$ (resp. $\chi$) (resp. $\psi$) transition at $v$ with respect to
$C$, then $P$ involves the $\psi$ (resp. $\chi$) (resp. $\phi$) transition
with respect to $C\ast v$. If $P$ involves the $\phi$ transition at $v$ with
respect to $C$, then according to Definition \ref{modint} the $v$ column of
$M(C,P)$ is $\vec{v}$ and the $v$ column of $M(C,P)_{\operatorname{mod}}^{v}$
is $\vec{v}+\vec{N}(v)$. As $P$ involves the $\psi$ transition at $v$ with
respect to $C\ast v$, the $v$ column of $M(C\ast v,P)$ is also $\vec{v}%
+\vec{N}(v)$. If $v$ involves the $\chi$ transition at $v$ with respect to
$C$, then $M(C,P)$ and $M(C,P)_{\operatorname{mod}}^{v}$ have the same $v$
column, namely $\vec{N}(v)$. This is also the $v$ column of $M(C\ast v,P)$,
because $P$ involves the $\chi$ transition at $v$ with respect to $C\ast v$.
If $P$ involves the $\psi$ transition at $v$ with respect to $C$, then
according to Definition \ref{modint} the $v$ column of $M(C,P)$ is $\vec
{v}+\vec{N}(v)$, so the $v$ column of $M(C,P)_{\operatorname{mod}}^{v}$ is
$\vec{v}$. As $P$ involves the $\phi$ transition at $v$ with respect to $C\ast
v$, the $v$ column of $M(C\ast v,P)$ is also $\vec{v}$.

Now consider one of the columns involved in $M_{11}$. This column corresponds
to a vertex $w$ that neighbors $v$ in $\mathcal{I}(C)$ and $\mathcal{I}(C\ast
v)$. As indicated in Figure \ref{lalgint3}, if $P$ involves the $\phi$ (resp.
$\chi$) (resp. $\psi$) transition at $w$ with respect to $C$, then $P$
involves the $\phi$ (resp. $\psi$) (resp. $\chi$) transition at $w$ with
respect to $C\ast v$. If $P$ involves the $\phi$ transition then $M(C,P)$,
$M(C,P)_{\operatorname{mod}}^{v}$ and $M(C\ast v,P)$ all have the same $w$
column, namely $\vec{w}$. Otherwise, the difference between the $w$ column of
$M(C,P)$ and the $w$ column of $M(C,P)_{\operatorname{mod}}^{v}$ is simply
that the diagonal entry is toggled; according to Definition \ref{modint}, this
is the same as the difference between the $w$ column of $M(C,P)$ and the $w$
column of $M(C\ast v,P)$.

Finally, consider one of the columns involved in $M_{12}$. This column
corresponds to a vertex $w$ that does not neighbor $v$ in $\mathcal{I}(C)$ and
$\mathcal{I}(C\ast v)$. It follows that $M(C,P)$, $M(C\ast v,P)$ and
$M(C,P)_{\operatorname{mod}}^{v}$ all have the same $w$ column.
\end{proof}

\begin{corollary}
\label{bigcor}Suppose $C$ and $C^{\prime}$ are two Euler systems of $F$. Then
$M(C^{\prime},C)$ is nonsingular and for every circuit partition $P$,
\[
M(C^{\prime},P)=M(C^{\prime},C)\cdot M(C,P).
\]

\end{corollary}

\begin{proof}
Consider the double matrix%
\[%
\begin{bmatrix}
I & M(C,P)
\end{bmatrix}
\]
where $I=M(C,C)$ is the identity matrix. According to Kotzig's theorem, it is
possible to obtain $C^{\prime}$ from $C$ using a finite sequence of $\kappa
$-transformations. Theorem \ref{thmone} tells us that after applying the
corresponding sequence of modified local complementations we will have
obtained the double matrix%
\[%
\begin{bmatrix}
M(C^{\prime},C) & M(C^{\prime},P)
\end{bmatrix}
.
\]
If $E$ is the product of the elementary matrices corresponding to the row
operations involved in the modified local complementations, then $M(C^{\prime
},C)=E\cdot I$ and $M(C^{\prime},P)=E\cdot M(C,P)$.
\end{proof}

We refer to the formula $M(C^{\prime},P)=M(C^{\prime},C)\cdot M(C,P)$ as
\emph{naturality} of modified interlacement matrices.

\begin{corollary}
\label{invcor}If $C$ and $C^{\prime}$ are Euler systems of $F$ then
\[
M(C,C^{\prime})=M(C^{\prime},C)^{-1}.
\]

\end{corollary}

\ It follows from Corollary \ref{bigcor} that all the modified interlacement
matrices of a circuit partition $P$ have the same right nullspace, i.e., the
space
\[
\ker M(C,P)=\{n\in GF(2)^{V(F)}\mid M(C,P)\cdot n=\mathbf{0}\}
\]
does not vary with $C$. As we will see in\ Theorem \ref{coreker}, $\ker
M(C,P)$ coincides with the \emph{core space} of $P$, defined as follows by
Jaeger \cite{J1}.

\begin{definition}
If $\gamma$ is a circuit of $F$ then the \emph{core vector} $core(\gamma)$ is
the element of $GF(2)^{V(F)}$ whose $v$ coordinate is 1 if and only if
$\gamma$ is singly incident at $v$, i.e., $\gamma$ includes precisely two of
the four half-edges incident at $v$. The \emph{core space }$core(P)$ is the
subspace of $GF(2)^{V(F)}$ spanned by the core vectors of circuits of $P$.
\end{definition}

Observe that $core(\gamma)=\mathbf{0}$ if and only if $\gamma$ is an Euler
circuit of a connected component of $F$.

%

\begin{figure}
[ptb]
\begin{center}
\includegraphics[
trim=2.673353in 8.571426in 1.340799in 1.340322in,
height=0.6175in,
width=3.2024in
]%
{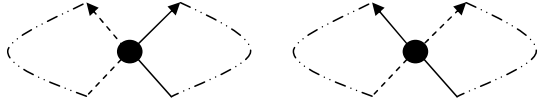}%
\caption{Two circuits are united, respecting orientations.}%
\label{lalgint4}%
\end{center}
\end{figure}

Here is a useful construction.

\begin{lemma}
Let $P$ be a circuit partition of a 4-regular graph $F$, which is not an Euler
system. Then there is an\ Euler system $C$ of $F$ with the following properties:

\begin{enumerate}
\item $P$ involves only $\phi$ and $\chi$ transitions with respect to $C$.

\item There is a circuit $\gamma_{0}\in P$ and a vertex $v_{0}$ incident on
$\gamma_{0}$, such that $P$ involves the $\chi$ transition with respect to $C$
at $v_{0}$, and $P$ involves the $\phi$ transition with respect to $C$ at
every other vertex incident on $\gamma_{0}$.

\item The core vector $core(\gamma_{0})$ is $\vec{v}_{0}+\vec{N}(v_{0})$,
where $\vec{v}_{0}\in GF(2)^{V(F)}$ is the vector whose only nonzero
coordinate corresponds to $v_{0}$, and $\vec{N}(v_{0})\in GF(2)^{V(F)}$ is the
vector whose $w$ coordinate is 1 if and only if $w$ neighbors $v_{0}$ in
$\mathcal{I}(C)$.
\end{enumerate}
\end{lemma}

\begin{proof}
We build an\ Euler system $C$ from $P$ as follows. For each circuit of $P$,
arbitrarily choose a preferred orientation. Find a vertex where two distinct
circuits of $P$ are incident, and let $P^{\prime}$ be the circuit partition
obtained by uniting the two incident circuits into one circuit, as indicated
in Figure \ref{lalgint4}. If $P^{\prime}$ is not an Euler system, find some
other vertex at which two distinct circuits of $P^{\prime}$ are incident, with
one of the two being a circuit of $P$; then unite them into one circuit as in
Figure \ref{lalgint4}. Repeat this process $\left\vert P\right\vert -c(F)$
times, at each step uniting two distinct circuits at least one of which is an
element of $P$. The process must end with an\ Euler system $C$. Observe that
at every vertex where two circuits are united during the construction, $P$
involves the $\chi$ transition with respect to $C$; at every other vertex, $P$
involves the $\phi$ transition with respect to $C.$

Let $v_{0}$ be the vertex at which two circuits are united in the last step of
the construction. Suppose that in the last step, a circuit $\gamma_{0}\in P$
is united with some other circuit at $v_{0}$. As $\gamma_{0}\in P$,
$\gamma_{0}$ must not have been involved in any earlier step. Consequently,
every vertex of $\gamma_{0}$ other than $v_{0}$ is a vertex where $P$ involves
the $\phi$ transition with respect to $C$. Also, one of the $v_{0}$-to-$v_{0}$
walks within the incident circuit of $C$ simply follows $\gamma_{0}$, so a
vertex $w\neq v_{0}$ neighbors $v_{0}$ in $\mathcal{I}(C)$ if and only if $w$
appears precisely once on $\gamma_{0}$.
\end{proof}

\begin{theorem}
\label{coreker}Let $P$ be a circuit partition of a 4-regular graph $F$, and
let $C$ be an Euler system of $F$. Then
\[
core(P)=\ker M(C,P).
\]

\end{theorem}

\begin{proof}
As $\ker M(C,P)$ does not vary with $C$, we need only prove that the theorem
holds for one choice of $C$. If $P$ itself is an Euler system, then the
theorem holds because $M(P,P)$ is the identity matrix and every core vector of
a circuit of $P$ is $\mathbf{0}$.

The proof proceeds by induction on $\left\vert P\right\vert >c(F)$. Let $C$,
$v_{0}$ and $\gamma_{0}$ be as in the lemma. Then the $v_{0}$ row of $M(C,P)$
is $\mathbf{0}$, and the $v_{0}$ column of $M(C,P)$ is $\vec{N}(v_{0})$.

Let $\gamma_{1}$ be the other circuit of $P$ incident at $v_{0}$, and let
$P^{\prime}$ be the circuit partition obtained from $P$ by uniting $\gamma
_{0}$ and $\gamma_{1}$ at $v_{0}$ as indicated in Figure \ref{lalgint4}. The
only difference between the transitions that appear in $P$ and the transitions
that appear in $P^{\prime}$ occurs at $v_{0}$, where $P$ involves the $\chi$
transition with respect to $C$ and $P^{\prime}$ involves the $\phi$ transition
with respect to $C$; hence the only difference between $M(C,P)$ and
$M(C,P^{\prime})$ is that the $v_{0}$ column of $M(C,P^{\prime})$ is $\vec
{v}_{0}$ and the $v_{0}$ column of $M(C,P)$ is $\vec{N}(v_{0})$. That is,
\[
M(C,P)=%
\begin{bmatrix}
0 & \mathbf{0} & \mathbf{0}\\
\mathbf{1} & I & A\\
\mathbf{0} & \mathbf{0} & B
\end{bmatrix}
\text{ and }M(C,P^{\prime})=%
\begin{bmatrix}
1 & \mathbf{0} & \mathbf{0}\\
\mathbf{0} & I & A\\
\mathbf{0} & \mathbf{0} & B
\end{bmatrix}
\]
where the first row and column correspond to $v_{0}$, $I$ is an identity
matrix involving the rows and columns corresponding to neighbors of $v_{0}$ in
$\mathcal{I}(C)$, and $B$ is a square matrix involving the rows and columns
corresponding to vertices $v\neq v_{0}$ that are not neighbors of $v_{0}$ in
$\mathcal{I}(C)$.

Notice that $M(C,P^{\prime})$ is row equivalent to
\[%
\begin{bmatrix}
1 & \mathbf{0} & \mathbf{0}\\
\mathbf{1} & I & A\\
\mathbf{0} & \mathbf{0} & B
\end{bmatrix}
,
\]
which differs from $M(C,P)$ only in one entry. Consequently the ranks of
$M(C,P^{\prime})$ and $M(C,P)$ do not differ by more than 1.

Considering the $v_{0}$ row of $M(C,P^{\prime})$, we see that every element of
$\ker M(C,P^{\prime})$ must have its $v_{0}$ coordinate equal to 0; clearly
then $\ker M(C,P^{\prime})\subseteq\ker M(C,P)$. Note also that $core(\gamma
_{0})=\vec{v}_{0}+\vec{N}(v_{0})\not \in \ker M(C,P^{\prime})$ and
$core(\gamma_{0})\in\ker M(C,P)$. The ranks of $M(C,P)$ and $M(C,P^{\prime})$
do not differ by more than 1, so
\[
\ker M(C,P)=\ker M(C,P^{\prime})+\left[  core(\gamma_{0})\right]
\]
where $\left[  core(\gamma_{0})\right]  $ denotes the one-dimensional subspace
spanned by $core(\gamma_{0})$.

As $\left\vert P^{\prime}\right\vert =\left\vert P\right\vert -1$, we may
assume inductively that $\ker M(C,P^{\prime})=core(P^{\prime})$. The core
vectors of the circuits of $P^{\prime}$ coincide with the core vectors of the
circuits of $P$, except for the fact that the core vector of the circuit
obtained by uniting $\gamma_{0}$ and $\gamma_{1}$ is $core(\gamma
_{0})+core(\gamma_{1})$. Consequently $\ker M(C,P)=\ker M(C,P^{\prime
})+\left[  core(\gamma_{0})\right]  $ is spanned by the core vectors of the
circuits of $P$ other than $\gamma_{0}$ and $\gamma_{1}$, together with
$core(\gamma_{0})+core(\gamma_{1})$ and $core(\gamma_{0})$. It follows that
$\ker M(C,P)=core(P)$.
\end{proof}

Theorem \ref{coreker} yields a useful formula with an interesting history; we
call it the \emph{circuit-nullity formula} \cite{Tbn}. Many special cases and
different versions of the formula have been discovered independently during
the last 100 years \cite{Be, BM, Bu, Br, CL, J1, Jo, KR, Lau, MP, Me, M, R,
So, S, Z}. In the notation of Theorem \ref{coreker}, the formula is%
\begin{equation}
\dim\ker M(C,P)=\left\vert P\right\vert -c(F). \tag{2.1}\label{circnull}%
\end{equation}

The proof is simple: Theorem \ref{coreker} implies that $\dim\ker M(C,P)=\dim
core(P)$, and Lemma \ref{coredim} implies that $\dim core(P)=|P|-c(F)$.

\begin{lemma}
\label{coredim}Let $P$ be a circuit partition of a 4-regular graph $F$, and
suppose $S\subseteq P$. Then the core vectors of elements of $S$ are linearly
independent if and only if there is no connected component of $F$ for which
$S$ contains all the incident circuits of $P$.
\end{lemma}

\begin{proof}
If $S$ includes all the circuits of $P$ incident on some connected component
of $F$ then the sum of the corresponding core vectors is $\mathbf{0}$, so the
set of core vectors of elements of $S$ is dependent.

Suppose $S$ does not include all the circuits of $P$ incident on any connected
component of $F$. In order to prove that the core vectors of the elements of
$S$ are linearly independent, we must show that for every nonempty subset $T$
of $S$, the sum of the core vectors of the elements of $T$ is nonzero. As
$T\subseteq S$, no element of $T$ is an Euler circuit for the corresponding
connected component of $F$, so the core vectors of the elements of $T$ are all
nonzero. If there is no vertex of $F$ at which two different circuits of $T$
are incident, then the core vectors of the elements of $T$ are all orthogonal
to each other, so their sum is certainly nonzero.

Suppose instead that $F$ has a vertex $v$ at which two different circuits of
$T$ are incident. Choose orientations for these two circuits, and let
$P^{\prime}$ be the circuit partition obtained from $P$ by using the operation
pictured in Figure \ref{lalgint4} at $v$; then $T$ gives rise to a
corresponding subset $T^{\prime}\subset P^{\prime}$. The operation unites two
circuits $\gamma_{1}$, $\gamma_{2}\in T$ into a single circuit $\gamma\in
T^{\prime}$ with $core(\gamma)=core(\gamma_{1})+core(\gamma_{2})$; the other
circuits of $T$ and $T^{\prime}$ are the same, so the sum of the core vectors
of the elements of $T$ is the same as the sum of the core vectors of the
elements of $T^{\prime}$. As $\left\vert P^{\prime}\right\vert <\left\vert
P\right\vert $, we may presume by induction that the sum of the core vectors
of the elements of $T^{\prime}$ is nonzero.
\end{proof}

\section{Discussion}

As mentioned in the introduction, the theory outlined in\ Section 2 includes
modified versions of ideas that have been known for decades. It seems that
interlacement first appeared in Brahana's version of the circuit-nullity
formula \cite{Br}, but the notion did not achieve broad recognition until the
1970s when interlacement was rediscovered in two areas of combinatorics:
Bouchet \cite{Bold} introduced the \emph{alternance} graph of an Euler circuit
of a 4-regular graph, and Cohn and Lempel \cite{CL} used \emph{link relation
matrices} to state a special case of the circuit-nullity formula in the
context of the theory of permutations. Shortly thereafter, Rosenstiehl and
Read \cite{RR} coined the term \emph{interlacement}, and used the technique to
analyze the problem of identifying the double occurrence words that correspond
to plane curves in general position (i.e., the only singularities are double
points). Since then, interlacement has given rise to the theory of circle
graphs \cite{Bec, Bu, Bco, F, Gh} and the more general theory of graph
equivalence under simple local complementation (see for instance Bouchet's
work on isotropic systems \cite{Bi1, Bi2}), and these ideas have inspired the
work of Arratia, Bollob\'{a}s and Sorkin on the interlace polynomials of
graphs \cite{A1, A2, A}. Also, as noted above several special cases and
different versions of the circuit-nullity formula have appeared in the
literature of combinatorics and low-dimensional topology during the last
thirty years, for instance in the work of Mellor \cite{Me}, Soboleva \cite{So}
and Zulli \cite{Z} on polynomial invariants of knots and links.

Only symmetric matrices appear in the references mentioned above. The same is
true of our earlier work \cite{T5}; the \emph{relative interlacement matrices}
we considered there are defined using a symmetric form of Definition
\ref{modint}, in which part 1 is replaced by the stipulation that the
off-diagonal entries of the rows and columns corresponding to $\phi$ vertices
are changed to 0. (The resulting matrix has the same $GF(2)$-nullity as
$M(C,P)$, so the circuit-nullity formula is valid under either definition.)
Corollary 20 of \cite{T5} states that under very particular circumstances,
there is a multiplicative relationship between the relative interlacement
matrices of a circuit partition with respect to two Euler systems. We
developed the modified interlacement machinery hoping to extend this
multiplicative naturality to arbitrary circuit partitions and Euler systems,
as in Corollary \ref{bigcor}. We are grateful to R. Brijder for pointing out
that matrices like the modified interlacement matrices appear in the
discussion of interlace polynomials given by Aigner and van der Holst
\cite{AH}; their Theorem 4 includes an implicit form of the circuit-nullity
formula (\ref{circnull}), but none of the other results we have presented
appear there.

We should also note that versions of Corollary \ref{invcor} have appeared in
the earlier literature: Jaeger \cite{J1} proved the special case in which the
Euler systems do not involve the same transition at any vertex, Bouchet
\cite{Bu} provided a different proof of the even more special case in which
the Euler systems involve only each others' $\chi$ transitions, and a general
result for relative interlacement matrices appeared in \cite{T5}. Without
naturality, though, the proofs of these results are considerably more
intricate than that of Corollary \ref{invcor}.


\bigskip

\begin{thebibliography}{99}                                                                                               %


\bibitem {AH}M. Aigner, H. van der Holst, Interlacement polynomials, Linear
Alg. Appl. 377 (2004) 11-30.

\bibitem {A1}R. Arratia, B. Bollob\'{a}s, G. B. Sorkin, The interlace
polynomial: A new graph polynomial, in: Proceedings of the Eleventh Annual
ACM-SIAM\ Symposium on Discrete Algorithms (San Francisco, CA, 2000), ACM, New
York, 2000, pp. 237-245.

\bibitem {A2}R. Arratia, B. Bollob\'{a}s, G. B. Sorkin, The interlace
polynomial of a graph, J. Combin. Theory Ser. B 92 (2004) 199-233.

\bibitem {A}R. Arratia, B. Bollob\'{a}s, G. B. Sorkin, A two-variable
interlace polynomial, Combinatorica 24 (2004) 567-584.

\bibitem {Be}I. Beck, Cycle decomposition by transpositions, J. Combin. Theory
Ser. A 23 (1977) 198-207.

\bibitem {BM}I. Beck, G. Moran, Introducing disjointness to a sequence of
transpositions, Ars. Combin. 22 (1986) 145-153.

\bibitem {Bold}A. Bouchet, Caract\'{e}risation des symboles crois\'{e}s de
genre nul, C. R. Acad. Sci. Paris S\'{e}r. A-B 274 (1972) A724-A727.

\bibitem {Bi1}A. Bouchet, Isotropic systems, European J. Combin. 8 (1987) 231-244.

\bibitem {Bec}A. Bouchet, Reducing prime graphs and recognizing circle graphs,
Combinatorica 7 (1987) 243-254.

\bibitem {Bi2}A. Bouchet, Graphic presentation of isotropic systems, J.
Combin. Theory Ser. B 45 (1988) 58-76.

\bibitem {Bu}A. Bouchet, Unimodularity and circle graphs, Discrete Math. 66
(1987) 203-208.

\bibitem {Bco}A. Bouchet, Circle graph obstructions, J. Combin. Theory Ser. B
60 (1994) 107-144.

\bibitem {Br}H. R. Brahana, Systems of circuits on two-dimensional manifolds,
Ann. Math. 23 (1921) 144-168.

\bibitem {BH2}R. Brijder, H. J. Hoogeboom, Nullity invariance for pivot and
the interlace polynomial, Linear Algebra Appl. 435 (2011) 277-288.

\bibitem {CL}M. Cohn, A. Lempel, Cycle decomposition by disjoint
transpositions, J. Combin. Theory Ser. A 13 (1972) 83-89.

\bibitem {F1}H. Fleischner, Eulerian graphs and related topics. Part 1. Vol.
1. Annals of Discrete Mathematics, 45. North-Holland Publishing Co.,
Amsterdam, 1990.

\bibitem {F2}H. Fleischner, Eulerian graphs and related topics. Part 1. Vol.
2. Annals of Discrete Mathematics, 50. North-Holland Publishing Co.,
Amsterdam, 1991.

\bibitem {F}H. de Fraysseix, A characterization of circle graphs, European J.
Combin. 5 (1984) 223-238.

\bibitem {Gh}L. Ghier, Double occurrence words with the same alternance graph,
Ars Combin. 36 (1993) 57--64.

\bibitem {IM1}D. P. Ilyutko, V. O. Manturov, Graph-links, in: Introductory
lectures on knot theory, (Trieste, Italy, 2009), World Scientific, New
Jersey-London, pp. 135-161.

\bibitem {J1}F. Jaeger, On some algebraic properties of graphs, in: Progress
in graph theory (Waterloo, Ont., 1982), Academic Press, Toronto, 1984, pp. 347-366.

\bibitem {Jo}J. Jonsson, On the number of Euler trails in directed graphs,
Math. Scand. 90 (2002) 191-214.

\bibitem {Kau}L. H. Kauffman, State models and the Jones polynomial, Topology
26 (1987) 395-407.

\bibitem {KR}J. Keir, R. B. Richter, Walks through every edge exactly twice
II, J. Graph Theory 21 (1996) 301-309.

\bibitem {K}A. Kotzig, Eulerian lines in finite 4-valent graphs and their
transformations, in: Theory of Graphs (Proc. Colloq., Tihany, 1966), Academic
Press, New York, 1968, pp. 219-230.

\bibitem {L2}M. Las Vergnas, On Eulerian partitions of graphs, in: Graph
theory and combinatorics (Proc. Conf., Open Univ., Milton Keynes, 1978), Res.
Notes in Math., 34, Pitman, Boston, Mass.-London, 1979, pp. 62--75.

\bibitem {L1}M. Las Vergnas, Eulerian circuits of 4-valent graphs imbedded in
surfaces, in: Algebraic methods in graph theory, Vol. I, II (Szeged, 1978),
Colloq. Math. Soc. J\'{a}nos Bolyai, 25, North-Holland, Amsterdam-New York,
1981, pp. 451--477.

\bibitem {L}M. Las Vergnas, Le polyn\^{o}me de Martin d'un graphe
Eul\'{e}rien, Ann. Discrete Math. 17 (1983) 397-411.

\bibitem {Lau}J. Lauri, On a formula for the number of Euler trails for a
class of digraphs, Discrete Math. 163 (1997) 307-312.

\bibitem {MP}N. Macris, J. V. Pul\'{e}, An alternative formula for the number
of Euler trails for a class of digraphs, Discrete Math. 154 (1996) 301-305.

\bibitem {Ma}P. Martin, Enum\'{e}rations eul\'{e}riennes dans les multigraphes
et invariants de Tutte-Grothendieck, Th\`{e}se, Grenoble (1977).

\bibitem {Me}B.\ Mellor, A few weight systems arising from intersection
graphs, Michigan Math. J. 51 (2003) 509-536.

\bibitem {M}G. Moran, Chords in a circle and linear algebra over GF(2), J.
Combin. Theory Ser. A 37 (1984) 239-247.

\bibitem {RR}R. C. Read, P. Rosenstiehl, On the Gauss crossing problem, in:
Combinatorics (Proc. Fifth Hungarian Colloq., Keszthely, 1976), Vol. II,
Colloq. Math. Soc. J\'{a}nos Bolyai, 18, North-Holland, Amsterdam-New York,
1978, pp. 843-876.

\bibitem {R}R. B. Richter, Walks through every edge exactly twice, J. Graph
Theory 18 (1994) 751-755.

\bibitem {So}E.\ Soboleva, Vassiliev knot invariants coming from Lie algebras
and 4-invariants, J. Knot Theory Ramifications 10 (2001) 161-169.

\bibitem {S}S. Stahl, On the product of certain permutations, Europ.
J.\ Combin. 8 (1987) 69-72.

\bibitem {Tbn}L.\ Traldi, Binary nullity,\ Euler circuits and interlacement
polynomials, Europ. J. Combinatorics 32 (2011) 944-950.

\bibitem {T5}L. Traldi, On the linear algebra of local complementation, Linear
Algebra Appl. 436 (2012) 1072--1089.

\bibitem {T6}L. Traldi, On the interlace polynomials, J. Combin. Theory Ser. B
(to appear).

\bibitem {Z}L.\ Zulli, A matrix for computing the Jones polynomial of a knot,
Topology 34 (1995) 717-729.
\end{thebibliography}
\end{document}